\documentclass[12pt,twoside]{amsart}
\usepackage{amsmath, amsthm, amscd, amsfonts, amssymb, graphicx}
\usepackage{enumerate}
\usepackage[colorlinks=true,
linkcolor=blue,
urlcolor=cyan,
citecolor=red]{hyperref}
\usepackage{mathrsfs}
\newtheorem{theorem}{Theorem}[section]
\newtheorem{lemma}{Lemma}[section]
\newtheorem{remark}{Remark}[section]

\newtheorem{proposition}{Proposition}[section]
\numberwithin{equation}{section}

\begin{document}
\title{Some remarks on Tsallis relative operator entropy}
\author{Shigeru Furuichi and Hamid Reza Moradi}
 
\maketitle

\begin{abstract}
This paper intends to give some new estimates for Tsallis relative operator entropy ${{T}_{v}}\left( A|B \right)=\frac{A{{\natural}_{v}}B-A}{v}$.
Let $A$ and $B$ be two positive invertible operators with the spectra contained in the interval $J \subset (0,\infty)$. We prove for any  $v\in \left[ -1,0 \right)\cup \left( 0,1 \right]$,
$$
(\ln_v t)A+\left( A{{\natural}_{v}}B+tA{{\natural}_{v-1}}B \right)\le {{T}_{v}}\left( A|B \right)  \le (\ln_v s)A+{{s}^{v-1}}\left( B-sA \right)  
$$
where $s,t\in J$. Especially, the upper bound for Tsallis relative operator entropy is a non-trivial new result. Meanwhile, some related and new results are also established. In particular, the monotonicity for Tsallis relative operator entropy is improved.

Furthermore, we introduce the exponential type relative operator entropies which are special cases of the perspective and we give inequalities among them and usual relative operator entropies.
\end{abstract}
%------------------------------------------------------------------------------------%
\pagestyle{myheadings}
\markboth{\centerline {Some remarks on Tsallis relative operator entropy}}
{\centerline {Shigeru Furuichi and Hamid Reza Moradi}}
\bigskip
\bigskip
%------------------------------------------------------------------------------------%
%------------------------------------------------------------------------------------%
\section{Introduction and Preliminaries}
Let $\mathcal{B}\left( \mathcal{H} \right)$ be the algebra of all (bounded linear) operators on a  complex Hilbert space $\mathcal{H}$. An operator $A$ on $\mathcal{H}$ is said to be positive (in symbol: $A\ge 0$) if $\left\langle Ax ,x  \right\rangle \ge 0$ for all $x \in \mathcal{H}$. We write $A>0$ if $A$ is positive and invertible. For self-adjoint operators $A$ and $B$, we write $A\ge B$ if $A-B$ is positive, i.e., $\left\langle Ax ,x  \right\rangle \ge \left\langle Bx ,x  \right\rangle $ for all $x \in \mathcal{H}$. We call it the usual order. In particular, for some scalars $m$ and $M$, we write $m\le A\le M$ if $m\left\langle x ,x  \right\rangle \le \left\langle Ax ,x  \right\rangle \le M\left\langle x ,x  \right\rangle $ for all $x \in \mathcal{H}$. The spectrum of $A$ is denoted by $Sp\left( A \right)$. A linear map $\Phi :\mathcal{B}\left( \mathcal{H} \right)\to \mathcal{B}\left( \mathcal{H} \right)$ is positive if $\Phi \left( A \right)\ge 0$ whenever $A\ge 0$. It is said to be unital if $\Phi \left( I \right)=I$.

%In what follows we denote the weighted arithmetic mean and the weighted geometric mean of positive operators $A$ and $B$ by $A{{\nabla }_{v}}B\equiv \left( 1-v \right)A+vB$ and $A{{\sharp}_{v}}B\equiv {{A}^{\frac{1}{2}}}{{\left( {{A}^{-\frac{1}{2}}}B{{A}^{-\frac{1}{2}}} \right)}^{v}}{{A}^{\frac{1}{2}}}$, respectively. For the case $v={1}/{2}\;$ we write $A\nabla B$ and $A\sharp B$.
In \cite{4}, for $A,$ $B>0$, relative operator entropy was defined by
\[S\left( A|B \right)={{A}^{\frac{1}{2}}}\ln \left( {{A}^{-\frac{1}{2}}}B{{A}^{-\frac{1}{2}}} \right){{A}^{\frac{1}{2}}}.\]
For $A,$ $B>0$, the Tsallis relative operator entropy is defined as follows (see \cite{5})
\[
{{T}_{v}}\left( A|B \right)=\frac{A{{\natural}_{v}}B-A}{v} ={{A}^{\frac{1}{2}}}\ln_v \left( {{A}^{-\frac{1}{2}}}B{{A}^{-\frac{1}{2}}} \right){{A}^{\frac{1}{2}}},\qquad\left(v\in \left[ -1,0 \right)\cup \left( 0,1 \right]\right),
\] 
where 
$\ln_v(x) := \frac{x^{v}-1}{v}$ is $v$-logarithmic function defined for $x>0$ with $0\neq v \in \mathbb{R}$
and
\[A{{\sharp}_{v}}B={{A}^{\frac{1}{2}}}{{\left( {{A}^{-\frac{1}{2}}}B{{A}^{-\frac{1}{2}}} \right)}^{v}}{{A}^{\frac{1}{2}}}\] 
is the weighted geometric mean for $v\in [0,1]$. 
Here we use the similar symbol $A\natural_v B={{A}^{\frac{1}{2}}}{{\left( {{A}^{-\frac{1}{2}}}B{{A}^{-\frac{1}{2}}} \right)}^{v}}{{A}^{\frac{1}{2}}}$ for any $v \in \mathbb{R}$. We remark that
\[{{T}_{0}}\left( A|B \right)=\underset{v\to 0}{\mathop{\lim }}\,{{T}_{v}}\left( A|B \right)=S\left( A|B \right),\]
since $\underset{v\to 0}{\mathop{\lim }}\,\ln_v t=\ln t$ for $t>0$. \\
It is known that \cite{6}:
\[A-A{{B}^{-1}}A\le S\left( A|B \right)\le B-A\quad\text{ and }\quad A-A{{B}^{-1}}A\le {{T}_{v}}\left( A|B \right)\le B-A.\]
Regarding the concavity of $f:\left[ m,M \right]\to \mathbb{R}$, we have
\begin{equation}\label{3}
\frac{M-t}{M-m}f\left( m \right)+\frac{t-m}{M-m}f\left( M \right)\le f\left( t \right)\quad\text{ for } t\in \left[ m,M \right].
\end{equation}
Since the function ${{\ln }_{v}}\left( t \right)=\frac{{{t}^{v}}-1}{v}$ on $t >0$ is  concave for the case $v \leq 1$, we have 
\begin{equation}\label{6}
\frac{\ln_v m}{M-m}\left(MA-B\right) +\frac{\ln_v M}{M-m}\left(B-mA\right) \leq T_v(A|B)
\end{equation}
whenever $mA\le B\le MA$ for some scalars $0<m<M$ and positive  invertible operators $A,B$. For the limit of $v\to 0$, \eqref{6} implies 
\begin{equation}\label{2}
\frac{\ln m}{M-m}\left( MA-B \right)+\frac{\ln M}{M-m}\left( B-mA \right)\le S\left( A|B \right).
\end{equation}

Motivated by the inequality \eqref{2}, Dragomir \cite{2} established some upper and  lower bounds for the quantity
\[S\left( A|B \right)-\frac{\ln m}{M-m}\left( MA-B \right)-\frac{\ln M}{M-m}\left( B-mA \right).\]

In this note‎, ‎we present some inequalities for Tsallis relative operator entropy, which are refinements and generalizations of some results obtained by Furuichi et al. \cite{6}. Some related and new inequalities, related to the information monotonicity of Tsallis relative operator entropy, are also established. We also give some alternative bounds for the recent result by Dragomir \cite{2}.

\section{Main Results}

At this point, for the reader's convenience,  we define the following two abbreviations:
\begin{equation}\label{1}
\xi \left( t \right)=1+\frac{{{2}^{\frac{t-m}{M-m}}}\left( t-m \right)\left( M-t \right){{M}^{\frac{t-M}{M-m}}}}{{{\left( M+m \right)}^{\frac{t+M-2m}{M-m}}}}\quad\text{ and }\quad \psi\left( t \right)=1+\frac{\left( t-m \right)\left( M-t \right){{M}^{\frac{t-M}{M-m}}}}{2{{m}^{\frac{t+M-2m}{M-m}}}}
\end{equation}
where $t\in \left[ m,M \right]$ with $0<m\le M$.
\begin{theorem}\label{a}
	Let $A,B\in \mathcal{B}\left( \mathcal{H} \right)$ be two positive invertible operators satisfying $mA\le B\le MA$ for some scalars $0<m<M$. Then
	\[\begin{aligned}
	0&\le {{A}^{\frac{1}{2}}}\ln \xi \left( {{A}^{-\frac{1}{2}}}B{{A}^{-\frac{1}{2}}} \right){{A}^{\frac{1}{2}}} \\ 
	& \le S\left( A|B \right)-\frac{\ln m}{M-m}\left( MA-B \right)-\frac{\ln M}{M-m}\left( B-mA \right) \\ 
	& \le {{A}^{\frac{1}{2}}}\ln \psi\left( {{A}^{-\frac{1}{2}}}B{{A}^{-\frac{1}{2}}} \right){{A}^{\frac{1}{2}}}.  
	\end{aligned}\]
\end{theorem}
\begin{proof}
	In the previous paper (see the estimate (2.6) in \cite{1}), the authors stated that if $0<b\le a$ and $v\in \left[ 0,1 \right]$, then
		\[{{m}_{v}}\left( \frac{b}{a} \right){{a}^{1-v}}{{b}^{v}}\le \left( 1-v \right)a+vb\le {{M}_{v}}\left( \frac{b}{a} \right){{a}^{1-v}}{{b}^{v}},\]
	where 
	\[{{m}_{v}}\left( \frac{b}{a} \right)=1+\frac{{{2}^{v}}v\left( 1-v \right){{\left( a-b \right)}^{2}}}{{{a}^{1-v}}{{\left( a+b \right)}^{1+v}}}\quad \text{ and }\quad {{M}_{v}}\left( \frac{b}{a} \right)=1+\frac{v\left( 1-v \right){{\left( a-b \right)}^{2}}}{2{{a}^{1-v}}{{b}^{1+v}}}.\] 
	Assume that $t\in \left[ m,M \right]$ with $0<m\le M$. Choosing $a=M$, $b=m$ and $v=\frac{t-m}{M-m}$, we get
	\[\xi \left( t \right){{M}^{\frac{t-m}{M-m}}}{{m}^{\frac{M-t}{M-m}}}\le t\le \psi\left( t \right){{M}^{\frac{t-m}{M-m}}}{{m}^{\frac{M-t}{M-m}}}\]
	where $\xi\left( t \right)$ and $\psi\left( t \right)$ are defined as in \eqref{1}. By taking the logarithm we obtain
	\[\ln \xi \left( t \right)\le \ln t-\frac{\ln m}{M-m}\left( M-t \right)-\frac{\ln M}{M-m}\left( t-m \right)\le \ln \psi\left( t \right).\] 
	Now, by setting $t={{A}^{-\frac{1}{2}}}B{{A}^{-\frac{1}{2}}}$ and then multiplying both sides by ${{A}^{\frac{1}{2}}}$ we deduce the desired result.
\end{proof}
\begin{remark}
Dragomir in \cite[Theorem 2]{2} has proved the following inequalities
\[\begin{aligned}
 0&\le K\left( \frac{M}{m} \right)\left( \frac{1}{2}A-\frac{1}{M-m}{{A}^{\frac{1}{2}}}\left| {{A}^{-\frac{1}{2}}}\left( B-\frac{M+m}{2}A \right){{A}^{-\frac{1}{2}}} \right|{{A}^{\frac{1}{2}}} \right) \\ 
& \le S\left( A|B \right)-\frac{\ln m}{M-m}\left( MA-B \right)-\frac{\ln M}{M-m}\left( B-mA \right) \\ 
& \le K\left( \frac{M}{m} \right)\left( \frac{1}{2}A+\frac{1}{M-m}{{A}^{\frac{1}{2}}}\left| {{A}^{-\frac{1}{2}}}\left( B-\frac{M+m}{2}A \right){{A}^{-\frac{1}{2}}} \right|{{A}^{\frac{1}{2}}} \right). \\ 
\end{aligned}\]
To compare Theorem \ref{a} and Dragomir's result, it is sufficient to compare the following inequalities for $0<x\le 1$:
\begin{equation}\label{4}
{{m}_{v}}\left( x \right){{x}^{v}}\le \left( 1-v \right)+vx\le {{M}_{v}}\left( x \right){{x}^{v}}
\end{equation}
where
\begin{equation}\label{7}
{{m}_{v}}\left( x \right)=1+\frac{{{2}^{v}}v\left( 1-v \right){{\left( x-1 \right)}^{2}}}{{{\left( x+1 \right)}^{v+1}}}\quad\text{ and }\quad{{M}_{v}}\left( x \right)=1+\frac{v\left( 1-v \right){{\left( x-1 \right)}^{2}}}{2{{x}^{v+1}}}
\end{equation}
and
\[{{K}^{r}}\left( x \right){{x}^{v}}\le \left( 1-v \right)+vx\le {{K}^{R}}\left( x \right){{x}^{v}},\]
where $v\in \left[ 0,1 \right]$, $r=\min \left\{ v,1-v \right\}$, $R=\max \left\{ v,1-v \right\}$ and $K\left( x \right)=\frac{{{\left( x+1 \right)}^{2}}}{4x}$. As we mentioned in \cite[Proposition 3.1]{1}, there is no ordering between ${{K}^{r}}\left( x \right)$ and ${{m}_{v}}\left( x \right)$ (and also between ${{K}^{r}}\left( x \right)$ and ${{M}_{v}}\left( x \right)$). Therefore we conclude that Theorem \ref{a} is not a trivial result.
\end{remark}
It seems worthwhile to point out the following remark.
\begin{remark}
Let $v\in \left( 0,1 \right]$. As we mentioned in \cite[Remark 1]{3}, for $x\ge 1$ we have
\[0\le 1-\frac{1}{x}\le {{\left( \frac{x+1}{2} \right)}^{v-1}}\left( x-1 \right)\le \frac{{{x}^{v}}-1}{v}\le \left( \frac{{{x}^{v-1}}+1}{2} \right)\left( x-1 \right)\le x-1.\]
It follows from \eqref{4} that
\begin{equation}\label{5}
1-\frac{1}{x}\le \frac{1}{{{M}_{v}}\left( x \right)}\left( \frac{1-v}{v}+x \right)-\frac{1}{v}\le \frac{{{x}^{v}}-1}{v}\le \frac{1}{{{m}_{v}}\left( x \right)}\left( \frac{1-v}{v}+x \right)-\frac{1}{v}\le x-1<0,	
\end{equation}
where $0<x\le 1$ and $0< v \leq 1$.

Actually, the first inequality in (\ref{5}) is equivalent to the inequality
$$
v(x-1)^2 g(v,x) \geq 0,\quad g(v,x) \equiv 2x^{v+1}-(1-v)\left\{(1+v)x -v\right\}
$$
so that we have only to prove $g(v,x) \geq 0$ for $0 < v \leq 1$ and $0< x \leq 1$.\\
Since $\frac{dg(v,x)}{dx} = (1+v)(2 x^v-1+v)$ and $\frac{d^2g(v,x)}{dx^2}=2v(1+v)x^{v-1} \geq 0$, we have minimum value of the function $g(v,x)$ when $x=\left( \frac{1-v}{2}\right)^{1/v}$ and it is calculated as
$$
g\left(v,\left( \frac{1-v}{2}\right)^{1/v} \right) =v(1-v)\left\{ 1-\left(\frac{1-v}{2} \right)^{1/v}\right\} \geq 0.
$$
%For the case of $x \geq 1$ and  $0 < v \leq 1$, it is easily seen that
%$\frac{dg(v,x)}{dx} \geq \frac{dg(v,1)}{dx} =(1+v)(1-v) \geq 0$ so that we have $g(v,x) \geq g(v,1) = 1+v \geq 0$.

The second and third inequalities in (\ref{5}) are straightforward from (\ref{4}).
The forth inequality in  (\ref{5}) is  equivalent to  the inequality for $x>0$ and $0<v \leq 1$,
$$
\frac{\left(m_v(x)-1\right)}{m_v(x)} \left(\frac{(1-v)+vx}{v}\right) \geq 0
$$
by simple calculations.
 
%From the above results, we also obtain the following inequalities
% \begin{equation}\label{5_02}
%0\leq 1-\frac{1}{x}\le \frac{1}{{{M}_{v}}\left( x \right)}\left( \frac{1-v}{v}+x \right)-\frac{1}{v}\le \frac{{{x}^{v}}-1}{v}\le \frac{1}{{{m}_{v}}\left( x \right)}\left( \frac{1-v}{v}+x \right)-\frac{1}{v}\le x-1,	
%\end{equation}
%where $x \ge 1$ and $0< v \leq 1$.

\end{remark}
\begin{remark}
We know that (see \cite[Remark 2.11]{moradi2}) if $x>0$ and $v\notin \left[ 0,1 \right]$, then 
\begin{equation}\label{15}
\left( 1-v \right)+vx\le {{x}^{v}}.
\end{equation}
The function ${{f}_{v}}\left( t \right)=\frac{v\left( 1-v \right)\left( t-1 \right)}{{{t}^{v+1}}}$ is convex for all $t>0$ and $v\in \left[ -1,0 \right]$. This follows from the following fact	
	\[f_{v}^{''}\left( t \right)=\frac{v\left( v+1 \right)\left( 1-v \right)\left( v\left( x-1 \right)-2 \right)}{{{x}^{v+3}}}\ge0.\]
	
Using the Hermite-Hadamard inequality, one can get
\[{{f}_{v}}\left( \frac{1+x}{2} \right)\le \frac{1}{1-x}\int_{x}^{1}{{{f}_{v}}\left( t \right)dt}\le \frac{{{f}_{v}}\left( 1 \right)+{{f}_{v}}\left( x \right)}{2}\]
 for each $0<x\le 1$ and $v\in \left[ -1,0 \right]$. The above inequality entails that
 \begin{equation}\label{16}
{{M}_{v}}\left( x \right)\le \frac{\left( 1-v \right)+vx}{{{x}^{v}}}\le {{m}_{v}}\left( x \right)
 \end{equation}
where ${{M}_{v}}\left( x \right)$ and ${{m}_{v}}\left( x \right)$ are defined as in \eqref{7}.

Notice that for $0<x\le 1$ and $v\in \left[ -1,0 \right]$, the functions ${{m}_{v}}\left( x \right)$ and ${{M}_{v}}\left( x \right)$ are increasing and $0<{{M}_{v}}\left( x \right)\le {{m}_{v}}\left( x \right)\le 1$ (similar to that of \cite[Proposition 2.4]{1}, so we omit details).
Thus, inequality \eqref{16} provides a refinement and a reverse for the inequality \eqref{15}.
\end{remark}
\begin{theorem}\label{10}
Let $A,B\in \mathcal{B}\left( \mathcal{H} \right)$ be two positive invertible operators such that $Sp\left( A \right),Sp\left( B \right)\subseteq J\subset (0,\infty)$. Then
$$
(\ln_v t)A+\left( A{{\natural}_{v }}B-tA{{\natural}_{v -1}}B \right)\le
{{T}_{v}}\left( A|B \right)\le (\ln_v s)A+{{s}^{v-1}}\left( B-sA \right),
$$ 
for any $s,t\in J$ and $v\in \left[ -1,0 \right)\cup \left( 0,1 \right]$.
\end{theorem}
\begin{proof}
The differentiable function $f$ is concave on $J$ if, for any $s\in J$, the tangent line through $\left( s,f\left( s \right) \right)$  is above the graph of $f$. That is
\begin{equation*}
f\left( t \right)\le f\left( s \right)+f'\left( s \right)\left( t-s \right).
\end{equation*}
Of course, the function $f\left( x \right)=\ln_v x\,\,\left( v\in \left[ -1,0 \right)\cup \left( 0,1 \right] \right)$ is differentiable and concave so
\begin{equation}\label{12}
\ln_v t \le \ln_v s+{{s}^{v-1}}\left( t-s \right),
\end{equation}
for any $s\in J$.\\
Applying functional calculus for the positive operator ${{A}^{-\frac{1}{2}}}B{{A}^{-\frac{1}{2}}}$, we get
$$
\ln_v \left( {{A}^{-\frac{1}{2}}}B{{A}^{-\frac{1}{2}}} \right)\le (\ln_v s)I+{{s}^{v-1}}\left( {{A}^{-\frac{1}{2}}}B{{A}^{-\frac{1}{2}}}-sI \right).
$$
for any $s\in J$. By multiplying both sides by ${{A}^{\frac{1}{2}}}$ we deduce the second inequality.

On the other hand, from inequality \eqref{12}, we get 
$$
(\ln_v t)I\le \ln_v \left( {{A}^{-\frac{1}{2}}}B{{A}^{-\frac{1}{2}}} \right)+ t{{\left( {{A}^{-\frac{1}{2}}}B{{A}^{-\frac{1}{2}}} \right)}^{v -1}}-{{\left( {{A}^{-\frac{1}{2}}}B{{A}^{-\frac{1}{2}}} \right)}^{v }} .
$$
Multiplying ${{A}^{\frac{1}{2}}}$ from both sides, we have the first inequality. 
\end{proof}

\begin{remark}
In \cite[Theorem 3.6]{6}, Furuichi et.al. obtained the following inequality:
\begin{equation}\label{eq01_remark2.4}
A\sharp_vB-\frac{1}{\alpha}A\natural_{v-1}B+\left(\ln_v\frac{1}{\alpha}\right)A \leq T_v(A|B)\leq \frac{1}{\alpha}B-A-\left(\ln_v\frac{1}{\alpha}\right)A\sharp_vB,
\end{equation}
for $\alpha >0$ and $0 < v \leq 1$.
We see the first inequality of \eqref{eq01_remark2.4} is just same to
one in Theorem \ref{10} with $t=\frac{1}{\alpha}$. However, its proof is an entirely different proof to  \cite[Theorem 3.6]{6}.

The second inequalities of  \eqref{eq01_remark2.4} and in Theorem \ref{10} with $s=\frac{1}{\alpha}$ are different. In order to compare these, we have to compare $g_v(s,t):=(\ln_v s)+s^{v-1}t-s^v$ and $h_v(s,t):=st-1-(\ln_v s)t^v$ for $v \in [-1,0)\cup (0,1]$ and $s,t >0$.
However there is no ordering between them. 
Actually, we have
$f_{0.5}(0.1,1) \simeq 1.01096$ and $f_{0.5}(0.1,0.1) \simeq -0.81$, where $f_v(s,t) :=g_v(s,t)-h_v(s,t)$. 
In addition our Theorem \ref{10} holds for $v\in[-1,0)\cup(0,1]$, while the inequalities \eqref{eq01_remark2.4} given in \cite[Theorem 3.6]{6} were shown for $v \in (0,1]$.
Therefore our obtained upper bound for $T_v(A|B)$ in Theorem \ref{10} is not trivial result. 

For the limit of $v \to 0$ in Theorem \ref{10} we have
$$
(\log t) A+(A-tAB^{-1}A) \leq S(A|B) \leq (\log s)A+s^{-1}B-A.
$$
In addition, if we take $t=\frac{1}{\alpha}$ and $s=\alpha$ in the above, then we obtain the known inequalities 
$$
(1-\log \alpha)A-\frac{1}{\alpha}AB^{-1}A\leq S(A|B) \leq (\log \alpha -1)A+\frac{1}{\alpha}.
$$
Furthermore, if we put $\alpha =1$ above, then we recover
$$
A-AB^{-1}A \leq S(A|B) \leq B-A.
$$
See \cite{6} for these known inequalities, for example.
\end{remark}

Extending a work of Fujii \cite{fujii} for relative operator entropy, Furuichi et al. \cite[Proposition 2.3]{6} obtained:
\begin{equation}\label{13}
\Phi \left( {{T}_{v}}\left( A|B \right) \right)\le {{T}_{v}}\left( \Phi \left( A \right)|\Phi \left( B \right) \right)
\end{equation}
where $\Phi :\mathcal{B}\left( \mathcal{H} \right)\to \mathcal{B}\left( \mathcal{H} \right)$ is a unital positive linear map. In the following, we try to improve inequality \eqref{13}, which is often called the monotonicity for Tsallis relative operator entropy.
\begin{theorem}\label{14}
Let $A,B\in \mathcal{B}\left( \mathcal{H} \right)$ be two positive invertible operators and let $\Phi :\mathcal{B}\left( \mathcal{H} \right)\to \mathcal{B}\left( \mathcal{H} \right)$ be a unital positive linear map. Then for any $v \in \left[ 0,1 \right]$
\[\begin{aligned}
 \Phi \left( {{T}_{v}}\left( A|B \right) \right)&\le \int_{0}^{1}\frac{ \Phi \left( \left( A{{\sharp}_{v}}B \right){{\sharp}_{\mu }}A \right){{\sharp}_{v}}\Phi\left(  \left( A{{\sharp}_{v}}B \right){{\sharp}_{\mu }}B \right)  -\Phi \left( A \right)}{v} d\mu\\ 
& \le {{T}_{v}}\left( \Phi \left( A \right)|\Phi \left( B \right) \right).
\end{aligned}\]
\end{theorem}
\begin{proof}
First, we need the following inequality from \cite[Theorem 3.1]{moradi2}, which connects both sides of the Ando's inequality \cite{ando}:

\begin{equation}\label{8}
\Phi \left( A{{\sharp}_{v}}B \right)\le \Phi\left(  \left( A{{\sharp}_{v}}B \right){{\sharp}_{\mu }}A \right){{\sharp}_{v}}\Phi \left(  \left( A{{\sharp}_{v}}B \right){{\sharp}_{\mu }}B \right)\le \Phi \left( A \right){{\sharp}_{v}}\Phi \left( B \right)
\end{equation}
for $v,\mu \in \left[ 0,1 \right]$.
 Integrating the inequality \eqref{8} over $\mu \in \left[ 0,1 \right]$, we obtain 
\begin{equation}\label{9}
\Phi \left( A{{\sharp}_{v}}B \right)\le \int_{0}^{1}{\Phi  \left(\left( A{{\sharp}_{v}}B \right){{\sharp}_{\mu }}A \right){{\sharp}_{v}} \Phi \left( \left( A{{\sharp}_{v}}B \right){{\sharp}_{\mu }}B \right) }d\mu \le \Phi \left( A \right){{\sharp}_{v}}\Phi \left( B \right).
\end{equation}
By virtue of \eqref{9}, we have
\[\begin{aligned}
 \frac{\Phi \left( A{{\sharp}_{v}}B \right)-\Phi \left( A \right)}{v}&\le \frac{\int_{0}^{1}{\Phi \left(  \left( A{{\sharp}_{v}}B \right){{\sharp}_{\mu }}A \right){{\sharp}_{v}} \Phi\left( \left( A{{\sharp}_{v}}B \right){{\sharp}_{\mu }}B \right) }d\mu -\Phi \left( A \right)}{v} \\ 
& \le \frac{\Phi \left( A \right){{\sharp}_{v}}\Phi \left( B \right)-\Phi \left( A \right)}{v}.  
\end{aligned}\]
 It follows from the linearity of $\Phi$ that $\Phi \left( \frac{A{{\sharp}_{v}}B-A}{v} \right)=\frac{\Phi \left( A{{\sharp}_{v}}B \right)-\Phi \left( A \right)}{v}$. This completes the proof.
\end{proof}

We close this section by giving a complementary inequality of \eqref{13}.
\begin{proposition}\label{prop_2.2}
Let $A,B\in \mathcal{B}\left( \mathcal{H} \right)$ be two positive invertible operators such that $Sp\left( A \right),Sp\left( B \right)\subseteq J$ and let $\Phi :\mathcal{B}\left( \mathcal{H} \right)\to \mathcal{B}\left( \mathcal{H} \right)$ be a unital positive linear map. Then  
\[\begin{aligned}
{{T}_{v}}\left( \Phi \left( A \right)|\Phi \left( B \right) \right)&\le \Phi \left( {{T}_{v}}\left( A|B \right) \right)+\left( \ln_v s-\ln_v t \right)\Phi \left( A \right) \\ 
&\quad +\Phi \left( tA{{\natural}_{v-1}}B-A{{\sharp}_{v}}B \right)+{{s}^{v-1}}\Phi \left( B-sA \right)  
\end{aligned}\]
for any $s,t\in J$ and $v\in \left[ 0,1 \right]$.
\end{proposition}
\begin{proof}
It follows from the first inequality in Theorem \ref{10} that
\begin{equation}\label{17}
(\ln_v t)\Phi \left( A \right)+\Phi \left( A{{\sharp}_{v}}B-tA{{\natural}_{v-1}}B \right)\le \Phi \left( {{T}_{v}}\left( A|B \right) \right).
\end{equation}
This implies
\begin{equation}\label{17_01}
0 \le \Phi \left( {{T}_{v}}\left( A|B \right) \right) -(\ln_v t)\Phi \left( A \right)-\Phi \left( A{{\sharp}_{v}}B-tA{{\natural}_{v-1}}B \right).
\end{equation}
On the other hand, the second inequality in Theorem \ref{10} implies
\begin{equation}\label{18}
{{T}_{v}}\left( \Phi \left( A \right)|\Phi \left( B \right) \right)\le (\ln_v s)\Phi \left( A \right)+{{s}^{v-1}}\Phi \left( B-sA \right).
\end{equation}
Adding  \eqref{17_01} to \eqref{18},  we infer the desired inequality.
\end{proof}

\begin{remark}
For the limit of $v\to 0$ in Proposition \ref{prop_2.2}, we obtain the inequality:
$$
S(\Phi(A)|\Phi(B)) \leq \Phi(S(A|B))+\left(\log \frac{s}{t}-2\right)\Phi(A) +\Phi(t AB^{-1}A-s^{-1}B).
$$
\end{remark}

\section{another look at Tsallis relative operator entropy}
Entropies are usually defined by the use of logarithmic functions. They can be redefined by the use of exponential functions in artificially formal. We study relative operator entropies defined by exponential functions and give some operator inequalities for them.
We here use $v$-exponential function defined by $\exp_v(x):=(1+vx)^{1/v}$ for $x>0$ and $v \in [-1,0)\cup (0,1]$. The function $\exp_v$ is the inverse of $\ln_v$, and we have $\lim_{v\to 0}\exp_v(x) = \exp(x)$.

It is known that Shannon entropy is defined by logarithmic function as $H({\bf p}) =- \sum_{j=1}^n p_j \log p_j$ for probability distribution ${\bf p} =(p_1,\cdots,p_n)$, where $p_j \geq 0$ and $\sum_{j=1}^n p_j=1$.
Putting $p_j = e^{-s_j}$, we can rewrite $H({\bf p})$ as $H^{exp}({\bf s}) = \sum_{j=1}^n s_j e^{-s_j}$, with $s_j \geq 0$. For Tsallis entropy $T_v({\bf p}) = \sum_{j=1}^n p_j \ln_v\frac{1}{p_j}$ can be rewritten as $T_v^{exp}({\bf s})=\sum_{j=1}^n s_j \exp_{-v}(-s_j)$ with $s_j \geq 0$, by putting $\frac{1}{p_j} = \exp_v(s_j)$. Similar modifications happen for the relative entropy $D({\bf p}|{\bf q}) = \sum_{j=1}^n p_j (\log p_j -\log q_j)$ and the Tsallis relative entropy $T_v({\bf p}|{\bf q}) = \sum_{j=1}^n p_j^{1-r}(\ln_v p_j -\ln_v q_j)$, for probability distribution ${\bf q} =(q_1,\cdots,q_n)$, where $q_j \geq 0$ and $\sum_{j=1}^n q_j=1$. By putting actually $p_j=e^{-s_j}$ and $q_j=e^{-t_j}$, we can rewrite the relative entropy as $D^{exp}({\bf s}|{\bf t}) = \sum_{j=1}^n (t_j-s_j)e^{-s_j}$. Putting by $\frac{1}{p_j} = \exp_v(s_j)$ and $\frac{1}{q_j} = \exp_v(t_j)$, we can  rewrite the Tsallis relative entropy as $D_v^{exp}({\bf s}|{\bf t}) = \sum_{j=1}^n \left\{ \left(\frac{\exp_{-v}(-t_j)}{\exp_{-v}(-s_j)}\right)^vt_j -s_j\right\}\exp_{-v}(-s_j)$. We note that 
$$\lim_{v \to 0} T_v^{exp}({\bf s}) = H^{exp}({\bf s})\quad  and \quad \lim_{v \to 0} D_r^{exp}({\bf s}|{\bf t}) =D^{exp}({\bf s}|{\bf t})$$ as known 
$\lim_{v \to 0} T_v({\bf p}) = H({\bf p})$ and $\lim_{v \to 0} T_v({\bf p}|{\bf q})=D({\bf p}|{\bf q})$. 
%
%For strictly positive operators $\rho$ and $\sigma$ with $Tr[\rho] = Tr[\sigma] =1$, these four quantities are written by (we use same symbols),
%$H^{exp}(\rho) = Tr[\rho e^{-\rho}]$, $T_r^{exp}(\rho)=Tr[\rho \exp_{-r}(-\rho)]$, $D^{exp}(\rho|\sigma)= Tr[(\sigma -\rho)e^{-\rho}]$ and $T_r^{exp}(\rho|\sigma) = Tr[\left\{\left(\exp_{-r}(-\rho)\right)^{-r} \left(\exp_{-r}(-\sigma)\right)^r \sigma -\rho\right\}\exp_{-r}(-\rho)]$.

As we have seen some entropies can be written by  $\exp(x)$ and  $\exp_v(x)$ without $\ln(x)$ and $\ln_r(x)$.
However we give nothing for information theoretical insights for our quantities. 
In this note, we just give mathematical inequalities for the relative operator entropies redefined by  $\exp(x)$ and $\exp_v(x)$. 
%Taking 
%\begin{equation}\label{03}
%\exp_v\left( t \right)\equiv {{\left( 1+vt \right)}^{\frac{1}{v}}},\quad \text{ }t>0,\text{ }v \in [-1,0)\cup(0,1].
%\end{equation}

We give scalar inequalities for the function $\exp_v(t)$ below.
\begin{lemma} \label{lemma01}
	Let $t>0$, then
	\[t\,{{\exp}_{{}^{v}/{}_{2}}}{{\left( t \right)}^{\frac{1-v}{2}}}\le {{\exp}_{v}}\left( t \right)-1\le \frac{t}{2} \left( 1+{{\exp}_{v}}{{\left( t \right)}^{1-v}} \right),\quad \left(-1\le v < 0\right),\] 
	\[\frac{t}{2} \left( 1+{{\exp}_{v}}{{\left( t \right)}^{1-v}} \right)\le {{\exp}_{v}}\left( t \right)-1\le t\, {{\exp}_{{}^{v}/{}_{2}}}{{\left( t \right)}^{\frac{1-v}{2}}},\quad \left(0 < v \le 1\right).\]
\end{lemma}
\begin{proof}
	Consider the function ${{f}_{v}}\left( t \right)={{\left( 1+vt \right)}^{\frac{1-v}{v}}}$ where $t>0$ and $v \in [-1,0)\cup(0,1]$. We have that 
	\begin{equation*}
	\left\{ \begin{array}{lr}
	f'{{'}_{v}}\left( t \right)\ge 0&\text{ }\left(-1\le v<0\right)\text{ } \\ 
	f'{{'}_{v}}\left( t \right)\le 0&\text{ }\left(0<v\le 1\right) \\ 
	\end{array} \right..
	\end{equation*}
	Now using the well-known Hermite-Hadamard inequality for convex (concave) function ${{f}_{v}}\left( t \right)$, we infer the desired result.
\end{proof}

Utilizing the definition of the $v$-exponential function $\exp_v(t)$ defined for $t>0$ and $v\in [-1,0)\cup(0,1]$, we can define
\[{{\mathcal{E}}_{v}}\left( A|B \right):= {{A}^{\frac{1}{2}}}{{\exp}_{v}}\left( {{A}^{-\frac{1}{2}}}B{{A}^{-\frac{1}{2}}} \right){{A}^{\frac{1}{2}}},\] 
where $A,B\in \mathcal{B}\left( \mathcal{H} \right)$ are two positive invertible operators. It is easy to see that ${{\lim }_{v\to 0}}{{\exp}_{v}}\left( t \right)={{\exp}{(t)}}$. So we have
\begin{equation}\label{04}
{{\lim }_{v\to 0}}{{\mathcal{E}}_{v}}\left( A|B \right)={{A}^{\frac{1}{2}}}{{\exp}({{{A}^{-\frac{1}{2}}}B{{A}^{-\frac{1}{2}}}}}){{A}^{\frac{1}{2}}}=:{{\mathcal{E}}}\left( A|B \right)
\end{equation}

It is notable ${{\mathcal{E}}}\left( A|B \right)$ and ${{\mathcal{E}}}_v\left( A|B \right)$ are special cases for the perspective \cite{ENG2011,EH2014}.
Lemma \ref{lemma01} implies the following operator inequalities.
\begin{theorem}
	Let $A,B\in \mathcal{B}\left( \mathcal{H} \right)$ be positive invertible operators.
	If $-1 \leq v < 0$, then
	$$
	BA^{-1} \mathcal{E}_{v/2}(A|B)^{\frac{1-v}{2}} + A \leq \mathcal{E}_{v}(A|B) \leq A+\frac{1}{2}B +\frac{1}{2}BA^{-1}\mathcal{E}_{v}(A|B)^{1-v}.
	$$
	If $0< v \leq 1$, then
	$$
	A+\frac{1}{2}B +\frac{1}{2}BA^{-1}\mathcal{E}_{v}(A|B)^{1-v} \leq \mathcal{E}_{v}(A|B) \leq BA^{-1} \mathcal{E}_{v/2}(A|B)^{\frac{1-v}{2}} + A.
	$$
\end{theorem}

We also have  the following relations among four operators.

\begin{theorem}
	Let $A,B\in \mathcal{B}\left( \mathcal{H} \right)$ be positive invertible operators, then
	\begin{equation}\label{01}
	S\left( A|B \right)\le {{T}_{v}}\left( A|B \right)\le {{\mathcal{E}}_{v}}\left( A|B \right)\le \mathcal{E}\left( A|B \right),\quad \text{ }\left(0<v\le 1\right),
	\end{equation}
	\begin{equation}\label{02}
	\mathcal{E}\left( A|B \right)\le {{\mathcal{E}}_{v}}\left( A|B \right)\le {{T}_{v}}\left( A|B \right)\le S\left( A|B \right),\quad \text{ }\left(-1\le v<0,\,\,\,v\ne -\frac{1}{2}\right),
	\end{equation}
	where $S\left( A|B \right)$ and ${{T}_{v}}\left( A|B \right)$ are relative operator entropy \cite{03} and Tsallis relative operator entropy \cite{02}, respectively.
\end{theorem}
\begin{proof}
	The first inequality in \eqref{01} and the last inequality in \eqref{02} follows directly from \cite[Proposition 3.1]{01}. We know that for each $t>0$,
	\[\exp_v(t)\leq\exp(t),\quad \text{ }\left(0<v\le 1\right),\] 
	\[\exp(t)\leq \exp_v(t),\quad \text{ }\left(-1\le v<0,\text{ }v\ne -\frac{1}{2}\right),\] 
	these gives the last and the first inequality in \eqref{1} and \eqref{02}, respectively. On the other hand for each $t>0$,
	\[\ln_v(t)\leq\exp_v(t),\text{ }\left(0<v\le 1\right)\] 
	\[\exp_v(t)\leq \ln_v(t),\text{ }\left(-1\le v<0,\,\,\,v\ne -\frac{1}{2}\right).\]
\end{proof}
\begin{theorem}
	Let $B\le C$. If $\frac{m-1}{v}A\le B\le \frac{M-1}{v}A$ with $0<m<M$, then
	\begin{equation}\label{05}
	{{\mathcal{E}}_{v}}\left( A|B \right)\le {{K}}\left( m,M,v \right){{\mathcal{E}}_{v}}\left( A|C \right),\quad \text{ }\left(v\in[-1,0)\cup(0,1]\right),
	\end{equation}
	where 
	\[{{K}}\left( m,M,v \right)\equiv \frac{\left( m{{M}^{v}}-M{{m}^{v}} \right)}{\left( v-1 \right)\left( M-m \right)}{{\left( \frac{\left( v-1 \right)\left( {{M}^{v}}-{{m}^{v}} \right)}{v\left( m{{M}^{v}}-M{{m}^{v}} \right)} \right)}^{v}}>0.\]
\end{theorem}
\begin{proof}
	According to \cite[Theorem 1.5]{04},
	\[\left\langle {{\left( I+v{{A}^{-\frac{1}{2}}}B{{A}^{-\frac{1}{2}}} \right)}^{\frac{1}{v}}}x,x \right\rangle \le {{K}}\left( m,M,v \right){{\left\langle \left( I+v{{A}^{-\frac{1}{2}}}B{{A}^{-\frac{1}{2}}} \right)x,x \right\rangle }^{\frac{1}{v}}},\] 
	holds for any unit vector $x\in \mathcal{H}$ and $v\in[-1,0)\cup(0,1]$.
	
	So we have
	\[\begin{aligned}
	\frac{1}{{{K}}\left( m,M,v \right)}\left\langle {{\left( I+v{{A}^{-\frac{1}{2}}}B{{A}^{-\frac{1}{2}}} \right)}^{\frac{1}{v}}}x,x \right\rangle &\le {{\left\langle \left( I+v{{A}^{-\frac{1}{2}}}B{{A}^{-\frac{1}{2}}} \right)x,x \right\rangle }^{\frac{1}{v}}} \\ 
	& \le {{\left\langle \left( I+v{{A}^{-\frac{1}{2}}}C{{A}^{-\frac{1}{2}}} \right)x,x \right\rangle }^{\frac{1}{v}}} \\ 
	& \le \left\langle {{\left( I+v{{A}^{-\frac{1}{2}}}C{{A}^{-\frac{1}{2}}} \right)}^{\frac{1}{v}}}x,x \right\rangle.
	\end{aligned}\]
	Whence
	\[{{\left( I+v{{A}^{-\frac{1}{2}}}B{{A}^{-\frac{1}{2}}} \right)}^{\frac{1}{v}}}\le {{K}}\left( m,M,v \right){{\left( I+v{{A}^{-\frac{1}{2}}}C{{A}^{-\frac{1}{2}}} \right)}^{\frac{1}{v}}}.\]
	Multiplying both sides by ${{A}^{\frac{1}{2}}}$ we deduce the desired inequality \eqref{05}.

\end{proof}

\section*{Acknowledgment}
The author (S.F.) was partially supported by JSPS KAKENHI Grant Number 16K05257.

\bibliographystyle{alpha}

\vskip 0.5 true cm 
{\tiny (S. Furuichi) Department of Information Science, College of Humanities and Sciences, Nihon University, 3-25-40, Sakurajyousui,	Setagaya-ku, Tokyo, 156-8550, Japan. }

{\tiny \textit{E-mail address:} furuichi@chs.nihon-u.ac.jp }

\vskip 0.3 true cm

{\tiny (H. R. Moradi) Department of Mathematics, Payame Noor University (PNU), P.O.Box, 19395-4697, Tehran, Iran.
	
	\textit{E-mail address:} hrmoradi@mshdiau.ac.ir}
%-----------------------------------------------------------------------------
%-----------------------------------------------------------------------------
\end{document}